\documentclass[11pt]{amsproc}
\usepackage{amssymb}

\usepackage{amsfonts,amsmath,amsthm}
\usepackage{amssymb}
\usepackage[utf8]{inputenc}
\usepackage[english]{babel}
\usepackage{graphicx}
\usepackage{epstopdf}
\usepackage{tikz}
\usepackage{hyperref}
\usetikzlibrary{shapes,arrows,backgrounds}
\usepackage{subcaption}
\usepackage{array}
\usepackage{url}
\usepackage{booktabs}

\newtheorem{theorem}{Theorem}[section]
\newtheorem*{theorem-non}{Theorem}
\newtheorem{lemma}[theorem]{Lemma}

\newtheorem{proposition}[theorem]{Proposition}
\newtheorem*{proposition-non}{Proposition}

\newtheorem{definition}[theorem]{Definition}

\newtheorem*{conjecture-non}{Conjecture}
\theoremstyle{remark}
\newtheorem{nota}[theorem]{}
\newtheorem{question}[theorem]{Question}
\newtheorem{example}[theorem]{Example}

\newtheorem{remark}[theorem]{Remark}

\def\cl{\overline}

\def\R{\mathbb R}
\def\C{\mathbb C}
\def\Q{\mathbb Q}
\def\K{\mathbb K}
\def\Z{\mathbb Z}
\def\N{\mathbb N}
\def\B{\mathcal B}

\def\Id{\operatorname{Id}}

\def\norma{\|\,\cdot\,\|}
\def\val{|\,\cdot\,|}
\def\nX{\|\,\cdot\,\|_X}
\def\nE{\|\,\cdot\,\|_E}

\def\codo{,\ldots, }
\def\ninf{\|\,\cdot\,\|_\infty}

\title{On the concept of non-ultrametric non-Archimedean analysis}

\begin{document}

\author{Javier~Cabello~Sánchez, Francisco~J.~Carmona~Fuertes}
\address{Departamento de Matem\'{a}ticas and Instituto de Matem\'{a}ticas. 
Universidad de Extremadura, Avda. de Elvas s/n, 06006 Badajoz. Spain}
\email{coco@unex.es, fjcf121@gmail.com}


\begin{abstract}
Given some non-Archimedean field $\K$ and some $\K$-linear space $X$, the usual 
way to define a norm over $X$ involves the {\em ultrametric inequality} 
$\|x+y\|\leq\max\{\|x\|,\|y\|\}$. 
In this note we will try to analyse the convenience of considering 
a wider variety of norms. 

The main result of the present note is a characterisation of the isometries 
between finite-dimensional linear spaces over some valued field endowed 
with the norm $\norma_1$, a result that can be seen as {\em the closest to a 
Mazur--Ulam Theorem in non-Archimedean analysis}. 
\end{abstract}

\maketitle

\section{Introduction}

\bigskip

Let $E=(\R^2,\nE)$ be a real normed plane. If $\nE$ comes from a inner product, 
then its unit sphere $S_E$ is an ellipse and so, it is determined by six points. 
This implies that, amongst all the norms that can be defined on $\R^2$ by means 
of a scalar product, $\nE$ is determined by 
$$\|(1,0)\|_E=\|(-1,0)\|_E,\ \ \|(0,1)\|_E=\|(0,-1)\|_E,\ \  
\|(1,1)\|_E=\|(-1,-1)\|_E.$$
With this in mind, the analysis of 2-dimensional real spaces could be summarized 
as {\em all of them are isometrically isomorphic and all of them have a huge amount 
of isometries}. In spite of this, the analysis of two-dimensional real 
normed spaces is a very interesting branch of modern mathematics; for example, 
\cite{SurveyMSW1, SurveyMSW2} are dedicated to the geometry of two-dimensional 
spaces and contain an overwhelming amount of results on these spaces and, moreover, 
the techniques employed to solve the 2-dimensional Tingley's Problem in the 
different cases have nothing in common, see~\cite{Banakh2, TBJCS}. 
This is so because $\R^2$ can be endowed with a huge variety of norms, with very 
different properties.

In $p$-adic analysis, we can find a very similar situation. Take some prime 
number $p\in\N$ and the usual valuation $|\cdot|:\Q_p\to[0,\infty)$ determined by the 
condition $|p|=1/p$. If we consider the linear space $\Q^2_p$ endowed with any 
{\em ultrametric} norm $\norma$ such that $\frac{\|(1,0)\|}{\|(0,1)\|}\not\in|\Q_p|$, 
equivalently, $\frac{\|(1,0)\|}{\|(0,1)\|}\not\in p^\Z$, then 
$\|(a,b)\|=\max\{|a|\|(1,0)\|,|b|\|(0,1)\|\}$ for every $(a,b)\in\Q^2_p$. 
This is so because of the strong triangular inequality 
$$\|(a,b)+(a',b')\|\leq\max\{\|(a,b)\|,\|(a',b')\|\}.$$

If we think about the initial development of the theory of Banach spaces, we 
can say that this theory flourished because of the necessity of studying the 
similarities and differences that $(C(K),\|\cdot\|_\infty)$ has with 
$(\ell_2,\norma_2)$ and, so to say, spaces of the same kind. In some sense, it 
could be said that these two spaces were the reason for mathematicians to 
start studying arbitrary norms over infinite dimensional spaces. 

In spite of this, $p$-adic analysis (and non-Archimedean analysis), seems 
to lack a good reason to develop a more general metric or normed framework. 
In~\cite[Proposition~1.1]{iwahori} it can be seen that, considering ultrametric 
norms, every finite-dimensional normed space $(\Q_p^n,\nX)$ over $\Q_p$ can 
be endowed with a basis $\{x_1,\ldots,x_n\}$ such that 
\begin{equation}\label{Xmax}
\|\lambda_1x_1+\ldots+\lambda_nx_n\|_X=
\max\{|\lambda_1|\|x_1\|_X,\ldots,|\lambda_n|\|x_n\|_X\}
\end{equation}
for every $\lambda_1,\ldots,\lambda_n\in\Q_p.$ The reader can find related results 
in~\cite[Theorem 50.8]{schikhoff} or \cite[p. 67]{vanRooij}. 

This result can be understood as follows: 

Given some $(\K,\val)$ and some $n\in\N$, one may take an $n$-tuple of positive 
real numbers $\alpha_1,\ldots,\alpha_n$ and endow $\K$ with the norms 
$\norma^1=\alpha_1\val,\ldots,\norma^n=\alpha_n\val$. Now we may consider 
$(X,\norma)=\oplus_\infty(\K,\norma^i)$, i.e., $X=\K^n$ and 
$\|(\lambda_1,\ldots,\lambda_n)\|=\max\{\alpha_1|\lambda_1|,\ldots,\alpha_n|\lambda_n|\}$. 
What~\cite[Proposition~1.1]{iwahori} says is that this is all in ultrametric 
analysis. Every $n$-dimensional normed space over a $p$-adic field is 
isometrically isomorphic to this construction. So to say, the only things that 
can vary are $\alpha_1,\ldots,\alpha_n$. 

\subsection{Definitions and preliminary results}

\begin{definition}\label{defVF}
A {\em valued field} is a field $\K$ equipped with a function
$|\,\cdot\,|:\K\to[0,\infty)$, called a {\em valuation}, such that
\begin{enumerate}
\item[i)] $|\lambda | = 0$ if and only if $\lambda = 0$,
\item[ii)] $|\lambda \mu| = |\lambda| |\mu|$,
\item[iii)] $|\lambda + \mu | \leq |\lambda |+ |\mu|$ for all $ \lambda , \mu \in \K$.
\end{enumerate}
\end{definition}

\begin{definition}\label{defNAF}
A {\em non-Archimedean} {valued field} is a valued field $(\K,\val)$ such that, 
for every $\alpha\in\K$, the map $\N\to[0,\infty)$ defined as $n\mapsto|n\alpha|$ 
is bounded. 
\end{definition}

\begin{definition}\label{defUMF}
An {\em ultrametric} {valued field} is a valued field $(\K,\val)$ such that 
\begin{enumerate}
\item[iii$'$)] $|\lambda + \mu | \leq max\{|\lambda |, |\mu|\}$ for all $ \lambda , \mu \in \K$. 
\end{enumerate}
\end{definition}

The following result, Ostrowski's Theorem, implies that 
every non-Archimedean valuation is actually ultrametric.

\begin{theorem}[\cite{schikhoff},10.1]\label{Ostrowski}
Each non-trivial valuation on the field of the rational numbers is equivalent 
either to the absolute value function or to some $p$-adic valuation.
\end{theorem}

So, the analysis of non-Archimedean {\em fields} reduces to that of ultrametric 
{\em fields}. But the situation does not need to be the same with 
non-Archimedean {\em normed spaces}. 

There may be something wrong about trying to develop {\em non-ultrametric 
non-Archimedean} analysis, but we have never found a reason to avoid the 
following definition of norm on non-Archimedean linear spaces. 

\begin{definition}\label{Defi}
Let $(\K,\val)$ be a valued field and $X$ a $\K$-linear space. 
A map $\norma:X\to[0,\infty)$ is a norm if, for every $\lambda\in\K$ and 
$x, y\in X$, it fulfils the following: 
\begin{enumerate}
\item $\|x\|=0$ if and only if $x=0$. 
\item $\|\lambda x\|=|\lambda|\|x\|.$
\item $\|x+y\|\leq\|x\|+\|y\|.$
\end{enumerate}
\end{definition}

Actually, Definition~\ref{Defi} appears in \cite[page EVT I.3]{Bourbaki}, 
where the author(s) do(es) not distinguish Archimedean from non-Archimedean 
fields. Moreover, they call {\em ultranorm} what is usally called {\em norm} 
in non-Archimedean analysis, see EVT I.26, Exercise 12 in~\cite{Bourbaki}. 
Maybe now, when $p$-adic differential geometry is undergoing such a rapid 
development (see, e.g, \cite{Fujiwara2018, Scholze}), we can begin considering 
a less restrictive normed structure over non-Archimedean fields. 

\begin{example}[Usual norm]\label{exampleinfty}
Given any valued field $(\K,\val)$ and any $n\in\N$, the map 
$\|\cdot\|_\infty:\K^n\to[0,\infty)$ defined as 
\begin{equation}\label{eqni}
\|(\lambda_1,\ldots,\lambda_n)\|_\infty=\max\{|\lambda_1|,\ldots,|\lambda_n|\} 
\end{equation}
is a norm. Actually, $\ninf$ not only fulfils the condition 
$\|x+y\|_\infty\leq\|x\|_\infty+\|y\|_\infty,$ but also the strong triangle inequality 
$\|x+y\|_\infty\leq\max\{\|x\|_\infty,\|y\|_\infty\}$. 
\end{example}

\begin{example}\label{norma1}
Consider any valued field $(\K,\val)$ and the map 
$\norma_1:\K^n:[0,\infty)$ defined as $\|(a_1\codo a_n)\|_1=|a_1|+\ldots|a_n|.$ 
Then, $\norma_1$ is a norm in the sense of Definition~\ref{Defi}. 
\end{example}

\begin{remark}\label{rabsolute}
From (\ref{Xmax}) it is clear that every {\em ultranorm} $\norma$ defined over 
$\Q^n_p$, fulfils the following property: \\
$\Q^n_p$ can be endowed with a basis $\B=\{x_1,\ldots,x_n\}$ such that, taking 
coordinates with respect to $\B$ one has that 
\begin{equation}\label{eqabsolute}
|\alpha_1|=|\beta_1|,\ldots,|\alpha_n|=|\beta_n| \text{\ \ implies\ \ }
\|(\alpha_1,\ldots,\alpha_n)\|=\|(\beta_1,\ldots,\beta_n)\|.
\end{equation}
In Archimedean analysis, a norm that fulfils (\ref{eqabsolute})  with respect to 
the usual basis is called absolute, see~\cite{tanakaR2}, we will say that 
$\norma$ is absolute with respect to $\B$ when (\ref{eqabsolute}) holds. 
\end{remark}

\begin{question}\label{qabsolute}
Is every non-Archimedean norm absolute with respect to some basis? 
\end{question}

\section{The main result}

\begin{nota}
In the sequel, we will use the word {\em isometry} with the meaning of 
{\em bijective map that preserves the distance}. 

Also, $\K$ will be a valued field and $\val$ will denote its valuation. 
\end{nota}

We will need the following result in the proof of Proposition~\ref{MU}.  

\begin{lemma}\label{oxy}
Let $x=(x_1,x_2),\ y=(y_1,y_2),\ z=(z_1,z_2)\in\K^2$, all of them 
different. Then, 
$$\|x-y\|_1=\|x-z\|_1+\|y-z\|_1 $$
if and only if $z$ shares one coordinate with $x$ and the other one with $y$, 
i.e., if and only if $x_1-z_1=y_2-z_2=0$ or $y_1-z_1=x_2-z_2=0$. 
\end{lemma}

\begin{proof}
As $\norma_1$ is translation invariant, we may suppose $z=(0,0)$ and 
denote $x=(x_1,x_2),y=(y_1,y_2)$. 
For $(x_1,x_2),(y_1,y_2)\in\K^2\setminus\{(0,0)\}$ we have the following: 
\begin{eqnarray}
\|(x_1,x_2)-(y_1,y_2)\|_1&\!\!=&\!\!|x_1-y_1|+|x_2-y_2|\stackrel{(1)}\leq
\max\{|x_1|,|y_1|\}+\max\{|x_2|,|y_2|\} \\
&&\stackrel{(2)}\leq|x_1|+|y_1|+|x_2|+|y_2|=\|(x_1,x_2)\|_1+\|(y_1,y_2)\|_1. \nonumber
\end{eqnarray}
The inequality $\stackrel{(1)}\leq$ holds because the valuation $\val$ fulfils 
the ultrametric inequality. Besides, the inequality $\stackrel{(2)}\leq$ 
is strict unless either $x_1=y_2=0$ or $y_1=x_2=0$, 
observe that any of these pairs of equalities imply that 
$$\|(x_1,x_2)-(y_1,y_2)\|_1=\|(x_1,x_2)\|_1+\|(y_1,y_2)\|_1.$$ 
So, the result holds. 
\end{proof}

\begin{proposition}\label{MU}
Let $(\K,\val)$ be a valued field and consider $\K^2$ endowed 
with $\norma_1$. Given $(\alpha,\beta)\in\K^2$ 
and isometries $\tau_1,\tau_2:(\K,\val)\to(\K,\val)$, the maps 
$\varphi,\psi:(\K^2,\norma_1)\to (\K^2,\norma_1)$ defined as 
$$\varphi(a,b)=(\alpha,\beta)+(\tau_1(a),\tau_2(b)) \text{\ \ and\ \ }
\psi(a,b)=(\alpha,\beta)+(\tau_2(a),\tau_1(b))$$
are isometries and every isometry arises one of these ways. 
\end{proposition}

\begin{proof}
Let $\tau_1,\tau_2$ and $(\alpha,\beta)$ be as in the statement. We need to show 
that for every $(a,b),(a',b')\in\K^2$ the maps $\varphi, \psi$ fulfil 
\begin{equation}\label{phipsi}
\|\varphi(a,b)-\varphi(a',b')\|_1=\|(a,b)-(a',b')\|_1=\|\psi(a,b)-\psi(a',b')\|_1.
\end{equation} 
On the one hand, the value of $(\alpha,\beta)$ does not affect any of the values 
that appear in (\ref{phipsi}), so we may suppose that $(\alpha,\beta)=(0,0)$. So, one has: 
\begin{eqnarray}\label{SonIso}
\|\varphi(a,b)-\varphi(a',b')\|_1= 
& & \|(\tau_1(a),\tau_2(b))-(\tau_1(a'),\tau_2(b'))\|_1= \nonumber \\
& & \|(\tau_1(a)-\tau_1(a'),\tau_2(b)-\tau_2(b'))\|_1= \nonumber \\
& & |\tau_1(a)-\tau_1(a')|+|\tau_2(b)-\tau_2(b')|\stackrel{*}= \\
& & |a-a'|+|b-b'|=\|(a,b)-(a',b')\|_1,\nonumber  
\end{eqnarray}
where the equality $\stackrel{*}=$ holds because $\tau_1$ and $\tau_2$ 
are isometries. 
With an analogous argument we see that $\psi$ is also an isometry. 

\medskip

Suppose, on the other hand, that $\phi:\K^2\to\K^2$ is an isometry such that 
$\phi(0,0)=(0,0).$ This readily implies that $\|\phi(a,b)\|_1=\|(a,b)\|_1$ for 
every $(a,b)\in\Q^p_2$. 

Given $(a,0),(0,b')$, as $\phi$ preserves norms and distances, we have 
$$\|\phi(a,0)-\phi(0,b')\|_1=\|\phi(a,0)\|_1+\|\phi(0,b')\|_1.$$
If we denote $(\cl{a},\cl{b})=\phi(a,0)$ and $(\cl{a'},\cl{b'})=\phi(0,b'),$ 
Lemma~\ref{oxy} implies that either $\cl{a}=\cl{b'}=0$ or $\cl{a'}=\cl{b}=0$ hold. 
This means that $\phi$ preserves the {\em horizontal} and {\em vertical} axes, 
in the sense that $\phi(\{(a,b):a=0\})\subseteq\{(a,b):a=0\}$ or 
$\phi(\{(a,b):a=0\})\subseteq\{(a,b):b=0\}$. As $\phi^{-1}$ is also an isometry, 
we obtain that either $\phi(\{(a,b):a=0\})=\{(a,b):a=0\}$ or 
$\phi(\{(a,b):a=0\})=\{(a,b):b=0\}$.
It is obvious that the map $(a,b)\mapsto(b,a)$ is an isometry, composing 
if necessary with this map we may suppose that $\phi$ maps each axis 
onto itself. So, there are $\tau_1,\tau_2$ such that 
$$\phi(a,0)=(\tau_1(a),0)\text{\ \ and\ \ }\phi(0,b)=(0,\tau_2(b)).$$
It is clear that both $\tau_1, \tau_2$ must be isometries, so we only need to 
show that $\phi(a,b)=(\tau_1(a),\tau_2(b))$ for every $(a,b)$. This is immediate 
from Lemma~\ref{oxy} because $(0,0)$ and $(\tau_1(a),\tau_2(b))$ are the 
only points that share one coordinate with $(\tau_1(a),0)$ and the other one 
with $(0,\tau_2(b))$, so we have finished the case $\phi(0,0)=(0,0).$ 

The general case is straighforward now, we only need to apply the previous 
case to the isometry given by the composition 
$$(a,b)\mapsto\phi(a,b)\mapsto(\phi(a,b)-\phi(0,0)).$$
\end{proof}

We will heavily use this Lemma in the proof Theorem~\ref{Kn1}. 

\begin{lemma}\label{ndimensional}
Let $(\K,\val)$ be a valued field and consider $(\K^n,\norma_1)$, 
$x=(x_1\codo x_n),$ $y=(y_1\codo y_n),$ $z=(z_1\codo z_n)\in \K^n$. Then, 
\begin{equation}\label{nxyz}
\|x-y\|_1=\|x-z\|_1+\|y-z\|_1
\end{equation}
if and only if $z_i\in\{x_i,y_i\}$ for every $i=1\codo n$. 
\end{lemma}

\begin{proof}
We may suppose $y=0,$ so (\ref{nxyz}) is equivalent to 
\begin{equation}\label{nxz}
\|x\|_1=\|x-z\|_1+\|z\|_1
\end{equation}
and we need to show that (\ref{nxz}) holds if and only if $z_i\in\{0,x_i\}$ 
for every $i=1\codo n$. 

Suppose that there is some $j$ for which $z_j\not\in\{0,x_j\}$. Then, one has 
$$|x_j|\leq\max\{|x_j-z_j|,|z_j|\}<|x_j-z_j|+|z_j|.$$
Taking into account that $\|(x_1\codo x_n)\|_1=|x_1|+\ldots+|x_n|$ we 
readily obtain that $\|x\|_1<\|x-z\|_1+\|z\|_1$. 

On the other hand, if every $z_i$ equals either 0 or $x_i$ then it is clear that 
$\|x\|_1=\|x-z\|_1+\|z\|_1$ because at each coordinate one has either 
$|x_i|=|x_i-x_i|+|x_i|$ or $|x_i|=|x_i-0|+|0|$. 
\end{proof}

\begin{theorem}\label{Kn1}
Let $(\K,\val)$ be a valued field and consider $(\K^n,\norma_1)$. 
If $\sigma\in S_n$ is a permutation and $\tau_1\codo \tau_n:\K\to\K$ 
are isometries, then $\phi:\K^n\to\K^n$, defined as 
\begin{equation}\label{FormaGeneral}
\phi(a_1,\ldots,a_n)=(\tau_1(a_{\sigma(1)}),\ldots, \tau_n(a_{\sigma(n)}))
\end{equation}
is a centred isometry and every centred isometry arises this way. 
\end{theorem}

\begin{proof}
It is clear that for every $\sigma\in S_n$ the map 
$\psi(a_1,\ldots,a_n)=(a_{\sigma(1)}\codo a_{\sigma(n)})$ 
is an isometry and repeating the computations (\ref{SonIso}) we obtain that 
$\varphi(a_1,\ldots,a_n)=(\tau_1(a_1)\codo\tau_n(a_n))$ is also an onto 
isometry provided that $\tau_1\codo \tau_n$ are isometries too. 
As the map $\phi$ given by (\ref{FormaGeneral}) is the composition of 
two isometries, it must be an isometry too. 

On the other hand, let $\phi:\K^n\to\K^n$ be an isometry. We will denote 
$\B=\{e_1\codo e_n\}$ the usual basis of $\K^n$. Consider $a\in\K$, 
$x=(a,0\codo 0)=ae_1$ and $\tau(x)$. Applying Lemma~\ref{ndimensional} to $x$ 
and $y=0$ implies that the only $z$ that fulfil 
$\|x\|_1=\|x-z\|_1+\|z\|_1$ 
are $z=0$ and $z=x$. As $\phi$ is an isometry, the only $z'$ that fulfil 
$\|\tau(x)\|_1=\|\tau(x)-z'\|_1+\|z'\|_1$ 
are $z'=\tau(0)=0$ and $z'=\tau(x)$. Now, Lemma~\ref{ndimensional} implies that 
$n-1$ coordinates of $\tau(x)$ are 0. So, we have again that $\phi$ preserves the 
union of all the axes. For $b\in\K\setminus\{0,a\}$ and $y=(b,0\codo 0)$ we may 
apply again Lemma~\ref{ndimensional} to obtain that $\tau(x)$ and $\tau(y)$ share 
exactly $n-1$ coordinates, so they must agree in the $n-1$ null coordinates that 
each of them has. This means that $\phi$ maps each axis into another axis: for 
every $i$ there exists $j$ such that 
\begin{equation}\label{ij}
\{\phi(\lambda e_i):\lambda\in\K\}\subset\{\lambda e_j:\lambda\in\K\}. 
\end{equation}
By symmetry, all of these inclusions are equalities, so there exists $\sigma\in S_n$ 
such that 
\begin{equation}\label{eqij}
\{\phi(\lambda e_i):\lambda\in\K\}=\{\lambda e_{\sigma(i)}:\lambda\in\K\} 
\text{ for every }i=1\codo n. 
\end{equation}
Composing if necessary with 
$\psi(a_1,\ldots,a_n)=(a_{\sigma^{-1}(1)}\codo a_{\sigma^{-1}(n)})$ 
we may suppose that $\sigma=\Id$, i.e., that every axis is invariant for $\phi$. 
This means that there exist isometries 
$\tau_1\codo \tau_n:\K\to\K$ such that $\phi(ae_i)=\tau_i(a)e_i$ for every 
$a\in\K,\ i=1\codo n$. 
We need to show that (\ref{FormaGeneral}) holds; as we are assuming that 
$\sigma=\Id$ we need to prove that 
\begin{equation}\label{ii}
\phi(a_1\codo a_n)=(\tau_1(a_1)\codo\tau_n(a_n)) \text{ for every }a_1\codo a_n\in\K. 
\end{equation}

In order to prove this we will use an inductive reasoning, take 
$x=(a_1\codo a_n)\in\K^n$ and let $k$ denote the number of nonzero coordinates 
of $x$. We have already seen that (\ref{ii}) holds when $k=1$, so suppose 
$1<k\leq n$ and that (\ref{ii}) is true for every vector with less that $k$ 
nonzero coordinates. 

As in (\ref{nxz}), applying Lemma~\ref{ndimensional} with $y=0$ we have 
\begin{equation}\label{xz}
\|x\|_1=\|x-z\|_1+\|z\|_1
\end{equation}
if and only if $z_i\in\{0,x_i\}$ for every $i$. 
As $n-k$ coordinates of $x$ are $0$, there exist exactly $2^k$ vectors in 
$\K^n$ that fulfil (\ref{xz}), it is clear that this implies that 
$\tau(x)$ has $k$ nonzero coordinates, too. Except for $x$, all the points 
that fulfil (\ref{xz}) have at most $k-1$ nonzero coordinates, so each of 
them fulfils (\ref{ii}). Let $i\in\{1\codo n\}$ such that $a_i\neq 0$. 
In particular, if we consider 
$$
z=a_ie_i\text{ and }z'=x-z=\displaystyle\sum_{j\neq i}a_je_j,
$$ 
then we have 
$$
\tau(z)=\tau_i(a_i)e_i,\ \tau(z')=\displaystyle\sum_{j\neq i}\tau_j(a_j)e_j. 
$$ 
As $\phi$ is a centred isometry and $z, z'$ fulfil (\ref{xz}) one has 
\begin{equation}\label{txz}
\|\tau(x)\|_1=\|\tau(x)-\tau(z)\|_1+\|\tau(z)\|_1;\quad 
\|\tau(x)\|_1=\|\tau(x)-\tau(z')\|_1+\|\tau(z')\|_1. 
\end{equation}
Applying again Lemma~\ref{ndimensional} we obtain that every nonzero coordinate 
of $\tau(z)$ or $\tau(z')$ agrees with the same coordinate of $\tau(x)$. 
Denoting $\tau(x)=b_1e_1+\cdots+b_ne_n$ we obtain $b_j=\tau_j(a_j)$ whenever 
$a_j\neq 0$. As we have noted before, the amount of nonzero coordinates of 
$\tau(x)$ must be $k$, so we are done. 
\end{proof}

\section{Concluding remarks}

\begin{remark}
There is a clear analogy between strict convexity of a normed real linear 
space and something that happens with $\norma_1$ in $\K^2$. If three points 
$a,b,c$ in a real, strictly convex, normed space fulfil 
$\|c-a\|=\|c-b\|+\|b-a\|$ then $b$ belongs to the segment whose endpoints 
are $a$ and $c$. Moreover, the distances $\|c-b\|$, $\|b-a\|$ determine $b$. 

If $a,b,c\in \K^2$ fulfil $\|c-a\|_1=\|c-b\|_1+\|b-a\|_1$ then $b$ shares 
some coordinate with $a$ and the other one with $c$. Moreover, if 
$|a_1-c_1|\neq|a_2-c_2|$ then $b$ is uniquely determined by $\|c-b\|_1$, $\|b-a\|_1$. 
Besides, $b$ is the point where the map $\K^2\to\R,$ $x\mapsto \|c-x\|_1+\|x-a\|_1$ 
attains its minimum, this may have some interest in non-Archimedean optimization, 
the interested reader may see~\cite{SHAMSEDDINE2001, SHAMSEDDINE200381}.

Also, in~\cite{MoslehianSadeghi} one can find a definition of 
non-Archimedean strictly convex space at the beginning of page 3406. 
Unfortunately, the conditions in that definition were too restrictive, as can 
be seen in~\cite{JCSJNGNLA}. It could be interesting to try another approach 
to strict convexity in non-Archimedean spaces based in this analogy. 
\end{remark}

\begin{remark}
There is no way to adapt Proposition~\ref{MU} to ultrametric normed spaces. 
Indeed, if $(E,\nE)$ is ultrametric and $\|e_0\|<\|v_0\|$ for some $e_0,v_0\in E$, 
then the map $T:E\to E$ defined as $T(v)=v$ if $\|v\|\neq\|v_0\|$, $T(v)=v+e_0$ 
if $\|v\|=\|v_0\|$ is an isometry and $T$ does not preserve axis or anything 
related. The fact that $T$ is an isometry can be seen in the proof of 
Proposition~1 in~\cite{Kubzdela}. A more detailed description of isometries 
between ultranormed spaces can be found in~\cite{JCSJNGIso}. 
\end{remark}

\begin{remark}
There is no way to improve Proposition~\ref{MU}, at least in spaces endowed 
with absolute norms. It is clear that any norm that fulfils~(\ref{eqabsolute}) is 
going to have at least the {\em axial isometries} described in Proposition~\ref{MU}. 
Anyway, this is not endemic to non-Archimedean spaces. It is true that every 
isometry between real Banach spaces is affine, but it is so because every 
isometry $\R\to\R$ is affine, too. Indeed, it is a well-known fact that $\R$-linear 
isometries between complex Banach spaces do not need to be $\C$-linear. 
Moreover, there exist isometric $\C$-linear spaces that are nor even 
$\C$-isomorphic (\cite{Bourgain}). 
\end{remark}

\section*{Acknowledgements}
The first author has been partially supported by DGICYT project PID2019-103961GB-C21 (Spain).

The first author wants to express his gratitude to José Navarro Garmendia and 
Juan B.~Sancho de Salas for many fruitful discussions on this topic. 

\bibliographystyle{abbrv}
\bibliography{NUNA}{}

\end{document}